\documentclass[12pt]{amsart}
\usepackage{amssymb,latexsym, amsmath, amscd, array, graphicx}
\def\Z {\mathbb{Z}}

\def\R {\mathbb{R}}
\def\C {\mathbb{C}}

\def \co{\colon\!}

\def \ftnote{\let\thefootnote\relax\footnotetext}

\def\cat{{\mathop\mathrm{cat}\,}}
\def\TC{{\mathop\mathrm{TC}\,}}

\newtheorem{theorem}{Theorem}
\newtheorem{problem}[theorem]{Problem}
\newtheorem{lemma}[theorem]{Lemma}
\newtheorem{corollary}[theorem]{Corollary}
\newtheorem{proposition}[theorem]{Proposition}

\theoremstyle{remark}
\newtheorem{remark}[theorem]{Remark}

\title[]{On LS-category and topological complexity of connected sum}
\author{Alexander Dranishnikov, Rustam Sadykov}
\date{}

\begin{document}
\maketitle

\begin{abstract} The Lusternik-Schnirelmann category and topological complexity are important invariants of manifolds (and more generally, topological spaces). We study the behavior of these invariants under the operation of taking the connected sum of manifolds. We give a complete answer for the LS-categoryof orientable manifolds, $\cat(M\# N)=\max\{\cat M,\cat N\}$. For topological complexity we prove the inequality $\TC  (M\# N)\ge\max\{\TC M,\TC N\}$ for simply connected manifolds.
\end{abstract}

\section{Introduction}

The (\emph{Lusternik-Schnirelmann) category}  $\cat X$ of a topological space $X$ is the least number $n$ such that there is a covering  $\{U_i\}$ of $X$ by $n+1$ open sets $U_i$ contractible in $X$ to a point. The category has many interesting and diverse applications. It  gives an estimate for the number of critical points of a function on a manifold~\cite{LS}, \cite{Co}, it was used to solve the Poincare problem on the existence of three closed geodesics on the sphere~\cite{LS29}, and it was used in establishing the Arnold conjecture for symplectic manifolds~\cite{Ru}.  For this reason, it is desirable to know the behavior of the invariant $\cat$  under elementary topological operations. In the present note we consider the operation of taking the connected sum of manifolds. 

The Lusternik-Schnirelmann category of the connected sum $M_1\sharp M_2$ of closed connected manifolds of dimension $n$ should be compared with \[
\cat(M_1\vee M_2) =\max\{\cat M_1, \cat M_2\}.
\] 
The first author proved~\cite{Dr14b} that $\cat(M_1\sharp M_2)$ is bounded above by $\cat(M_1\vee M_2)$. The same estimate under an additional assumption that $\cat M_1$ and $\cat M_2$ are at least $3$ is established in \cite{Ne}; the estimate also follows from \cite[Theorem 1]{FHT} with the assumption  that the manifolds $M_1$ and $M_2$ are simply connected. We show that the upper bound is exact.

\begin{theorem}\label{th:1} There is an equality
$\cat(M_1\sharp M_2)=\max\{\cat M_1, \cat M_2\}$ for closed connected orientable manifolds $M_1$ and $M_2$. 
\end{theorem}
Rudyak's conjecture~\cite{Ru2}
states that for a mapping of degree one $f:N\to M$ between oriented manifolds, $\cat N\ge \cat M$.
Theorem~\ref{th:1} proves Rudyak's conjecture for the collapsing maps $f:M_1\sharp M_2\to M_1$.

Given a topological space $X$, a \emph{motion planning algorithm} over an open subset $U_i\subset X\times X$ is a continuous map $U_i\to X^{[0,1]}$ that takes a pair $(x,y)$ to a path $s$ with end points $s(0)=x$ and $s(1)=y$. 
The \emph{topological complexity} $\TC(X)$ of  $X$ is the least number $n$  such that there is a covering $\{U_i\}$ of $X\times X$ by $n+1$ open sets over which there are motion planning algorithms. Topological complexity is motivated by problems in robotics, but it also has non-trivial relation to interesting problems in algebraic topology. For example, by the Farber-Tabachnikov-Yuzvinsky theorem \cite{FTY}, for $n\ne 1,3,7$, the topological complexity of $\R P^n$ is the least integer $k$ such that $\R P^n$ admits an immersion into $\R^k$. 

In the case of simply connected manifolds we give a lower bound for the topological complexity of connected sum similar to that for the category.
\begin{theorem}\label{th:0.2} For closed $r$-connected orientable manifolds $M_1$ and $M_2$, $r>0$, there is the inequality $$\TC (M_1\sharp M_2)\ge\max\{\TC M_1,\TC M_2\}.$$
\end{theorem}

Again it is reasonable to compare the topological complexity of $M_1\sharp M_2$ with the number $\TC(M_1\vee M_2)$. In view of the equality $\TC(X)=\TC^M(X)$ for  complexes with $\TC(X)>\dim X$,
\cite[Theorem 2.5]{Dr14}, it follows \cite[Theorem 3.6]{Dr14},   that 
\[
   \max\{\TC M_1, \TC M_2, \cat (M_1\times M_2)\}\le  \TC(M_1\vee M_2)\le \TC(M_1)+\TC(M_2).
\]

It is natural to expect that $\TC (M_1\vee M_2)$ majorates $\TC (M_1\# M_2)$. We proved it only under some conditions.
\begin{theorem}\label{th:0.3} For closed $r$-connected orientable $n$-manifolds $M_i$  with $\TC M_i\ge\frac{n+2}{r+1}$, $i=1,2$, 
there is the inequality $$\TC (M_1\vee M_2)\ge\TC (M_1\# M_2).$$
\end{theorem}
For simply connected manifolds this upper bound was established by Calcines-Vanderbroucq \cite{CV}  under hypotheses that are slightly weaker than those of Thorem~\ref{th:0.3}.

\subsubsection*{Organization of the paper} In section~\ref{s:1} we prove preliminary statements that are necessary for proofs of Theorem~\ref{th:1} and \ref{th:0.2}. The three theorems are proved in sections ~\ref{s:2}, \ref{s:3}, and \ref{s:4} respectively. Finally, in section~\ref{s:5} we present several low dimensional examples.

\section{Fiberwise join product}

Recall that an element of an iterated join $X_0*X_1*\cdots*X_n$ of topological spaces is a formal linear combination $t_0x_0+\cdots +t_nx_n$ of points $x_i\in X_i$ with $\sum t_i=1$, $t_i\ge 0$, in which all terms of the form $0x_i$ are dropped. Given fibrations $f_i\co X_i\to Y$ for $i=0, ..., n$, the fiberwise join of spaces $X_0, ..., X_n$ is defined to be the space
\[
    X_0*_YX_1*_Y\cdots *_YX_n=\{\ t_0x_0+\cdots +t_nx_n\in X_0*\cdots *X_n\ |\ f_0(x_0)=\cdots =f_n(x_n)\ \}.
\]
The fiberwise join of fibrations $f_0, ..., f_n$ is the fibration 
\[
    f_0*_Y*\cdots *_Yf_n\co X_0*_YX_1*_Y\cdots *_YX_n \longrightarrow Y
\]
defined by taking a point $t_0x_0+\cdots +t_nx_n$ to $f_i(x_i)$ for any $i$. As the name `fiberwise join' suggests, the fiber of the fiberwise join of fibrations is given by the join of fibers of fibrations. 

When $X_i=X$ and $f_i=f:X\to Y$ for all $i$  the fiberwise join of spaces is denoted by $*^{n+1}_YX$ and the fiberwise join of fibrations is denoted by $*_Y^{n+1}f$. 

For a topological space $X$, we turn an inclusion of a point $*\to X$ into a fibration $G^0_X\to X$ and the diagonal inclusion $X\to X\times X$ into a fibration $\Delta^0_X$. The $n$-th Ganea space of $X$ is defined to be the space $G^n_X=*_X^{n+1}G^0_X$, while the $n$-th Ganea fibration is the fiberwise join of fibrations $G^n_X\to X$. Similarly, there is a fiberwise join $\Delta^n_X=*^{n+1}_{X\times X}\Delta^0_X$ and a fiberwise join fibration $\Delta^n_X\to X\times X$. 

The Schwarz theorem~\cite{Sch} implies that $\cat(X)\le n$ if and only if the fibration $G^n_X\to X$ admits a section. Similarly, $\TC(X)\le n$ if and only if the fibration $\Delta^n_X\to X$ admits a section.

\section{Preliminary results}\label{s:1}

Let $M_1$ and $M_2$ be two closed connected manifolds of dimension $n$.  To simplify notation, we will write $M_\vee$ and $M_\sharp$ respectively for the pointed union $M_1\vee M_2$ and the connected sum $M_1\sharp M_2$. Recall that a map of topological spaces is an \emph{$n$-equivalence} if it induces an isomorphism of homotopy groups in degrees $\le n-1$ and an epimorphism in degree $n$.

\begin{proposition}\label{p:1} The projection of the connected sum $M_\sharp$ to the pointed sum $M_\vee$ is an $(n-1)$-equivalence. 
\end{proposition}
\begin{proof}
 By replacing $M_\vee$ with the mapping cylinder of the projection of the connected sum onto the pointed union, we may assume that the projection is an inclusion.  It follows that $H_i(M_\vee, M_\sharp)=0$ for $i\le n-1$, and the pair $(M_\vee, M_\sharp)$ is simply connected. By the relative Hurewicz theorem, then, all homotopy groups of the pair in degrees $\le n-1$ are trivial.
\end{proof}

We omit the proof of the following observation as it is straightforward. 

\begin{proposition}\label{p:4} Let $f\co X\to Y$ be an $(n-1)$-equivalence of CW complexes. Then there is a CW complex $Y'$ containing $X$ with its $(n-1)$-skeleton in $X$, and a homotopy equivalence $j\co Y\to Y'$ such that $j\circ f$ is homotopic to the inlcusion $X\subset Y'$. \end{proposition}

We will denote the join of $k+1$ copies of a space $X$ by $*^{k+1}X$. 

For $r=0$ the following Proposition was proven in \cite[Proposition 5.7]{DKR}.
 Also a slightly weaker statement for $r\ge 1$ was proven  in~\cite{FHT}.

\begin{proposition} \label{p:2} Let $f\co X\to Y$ be an $(n-2)$-equivalence of $(r-1)$-connected pointed CW-complexes with $r\ge 0$. Then the map $*^{k+1}f$ is a $(kr+k+n-2)$-equivalence. 
\end{proposition}
\begin{proof} By Proposition~\ref{p:4}, we may assume that $X$ is a CW subcomplex of $Y$ and that the $(n-2)$-skeleton of $Y$ belongs to $X$. Given a CW complex $X$, we may introduce a CW complex structure on $*^{k+1}X$ by defining the cells in the join to be the joins $D=D_1*\cdots *D_{k+1}$ of cells in $X$. We note that if $\dim D_i=d_i$, then $\dim D=k+\sum d_i$.  

Recall that the join of spaces is homotopy equivalent to the reduced join of spaces. Hence we may assume  that all joins are reduced. Then the complex $*^{k+1}Y$ does not have cells of dimension $\le kr+k+n-3$ that are not in $*^{k+1}X$. Hence, the map $*^{k+1}$ is a $(kr+k+n-2)$-equivalence. 
\end{proof}




We recall that using the orientation sheaf ${\mathcal O}_M$ one a closed manifold $M$ one can define a fundamental class $[M]$ such that the Poincare Duality
homomorphism $$PD=[M]\cap :H^k(M;\mathcal F)\to H_{n-k}(M;\mathcal F\otimes{\mathcal O}_M)$$ is an isomorphism~\cite{Bre} for any locall coefficients $\mathcal F$ on $M$.
If a map $f:M\to N$ between closed manifolds takes the orientation sheaf ${\mathcal O}_N$ to the orientation sheaf ${\mathcal O}_M$, then one can define the degree of $f$ as the integer $deg(f)$ such that $f_*([M])=deg(f)[N]$.

\begin{lemma}\label{l:6}  Suppose that $f:M\to N$ is a map of degree one between closed manifolds. Then for any local coefficients $\mathcal F$ the homomorphism
$f^*:H^k(N;\mathcal F)\to H^k(M;f^*\mathcal F)$ is a monomorphism for all $k$ and it is an isomorphism for $k=n$.
\end{lemma}
\begin{proof}
This fact is well-known for orientable $M$ and $N$. The same argument works in the general case. Namely, the equality $$f^!f^*(\alpha)=PD^{-1}f_*([M]\cap f^*(\alpha))=PD^{-1}([N]\cap\alpha)=\alpha$$ for $\alpha\in H^k(N;\mathcal F)$ and $f^!:H^k(M;f^*\mathcal F)\to H^k(N;\mathcal F)$ defined as $f^!=PD^{-1}f_*PD$ implies that $f^*$ is a section of $f^!$. Thus, $f^*$ is a monomorphism. 

Let $A$ be the $\pi_1(N)$-module that corresponds to $\mathcal F$. It is treated as $\pi_1(M)$-module for (co)homology of $M$ with coefficients in $A$.
In  the commutative diagram
\[
\begin{CD}
H^n(M;A) @<f^*<< N^n(N;A)\\
@V{PD}VV @V{PD}VV\\
H_0(M;A\otimes\Z)=(A\otimes\Z)_{\pi_1(M)} @>=>> H_0(N;A\otimes\Z)=(A\otimes\Z)_{\pi_1(N)}
\end{CD}
\]
the bottom arrow is an isomorphism of the coinvariants, since  a degree one map induces an epimorphism of the fundamental groups
and the diagonal action of $\pi_1(M)$ on $A\times\Z$ factors through that of $\pi_1(N)$. Here the action of $\pi_1(N)$ on $\Z$ corresponds to the orientation sheaf.
\end{proof}



\section{Proof of Theorem~\ref{th:1}}\label{s:2}

Since a closed manifold of category $\le 1$ is homeomorphic to a sphere~\cite{CLOT}, in the rest of the argument we may assume that the categories of $M_1$ and $M_2$ are at least $2$.  

We may assume that $M_\vee$ is obtained from  $M_\sharp$ by attaching a disc along the neck of the connected sum. Hence, we may choose a  CW-structure on $M_\vee$ so that $M_1, M_2$ and $M_\sharp$ are CW subcomplexes of $M_\vee$, and the $(n-1)$-skeleton of $M_\vee$ is in $M_\sharp$. Let $j\co M_\sharp \to M_\vee$ denote the inclusion, and $p_i\co M_\vee\to M_i$ denote the retraction to the $i$-th pointed summand for $i=1,2$. The inclusion $j$ of $M_\sharp$ into $M_\vee$ induces an inclusion $j_G$ of their $k$-th Ganea fibrations $G^k_\sharp\to G^k_\vee$, which, in its turn, restricts to an inclusion $j_G'$ of fibers.

\begin{lemma}\label{eq} 
Suppose that $k\ge 2$. Then the map $j_G'$ of fibers of  the $k$-th Ganea fibrations is an $n$-equivalence. 
\end{lemma}
\begin{proof}
By Proposition~\ref{p:1}, the map $\Omega M_\sharp\to \Omega M_\vee$ is an $(n-2)$-equivalence of CW complexes. Therefore, by Proposition~\ref{p:2} the map $j_G'$  is an  $(kr+k+n-2)$-equivalence, where $(r-1)$ is the connectivity of $M_1$ and $M_2$. Since  $k\ge 2$, the map $j_G'$ is an $n$-equivalence. 
\end{proof}

Recall that the $k$-th Ganea fibration $G^k_X$ over a topological space  $X$ admits a section if and only if $\cat X\le k$. 

\begin{lemma}\label{l:1.1} Suppose that the Ganea fibration $G^k_\vee$ over $M_\vee$ admits a section $s_\vee$. Then the Ganea fibration $G^k_\sharp$ over $M_\sharp$ also admits a section. In particular, $\cat M_\sharp\le \cat M_\vee$. 
\end{lemma}
\begin{proof} Newton's Theorem 3.2~\cite{Ne} proves exactly that. In the pull-back diagram
\[
\begin{CD}
G^k_\sharp @>h>> E @>j'>> G^k_\vee\\
@. @Vp'VV @VpVV\\
@. M_\sharp @>j>> M_\vee\\
\end{CD}
\]
the map $h$ is an $n$-equivalence in view of Lemma~\ref{eq}. The section $s_\vee$ defines a section $s':M_\sharp\to E$.  Therefore,  since $\dim M_\#\le n$, the map $s'$ admits a lift $s''$ with respect to $h$. Then $s''$ is a section of the Ganea fibration $p_\sharp:G^k_\sharp\to M_\sharp$.
\end{proof}

\begin{lemma}\label{l:1.2} For connected closed manifolds $M_1$ and $M_2$ with $M_2$ orientable,  $$\cat M_1\sharp M_2\ge\cat M_1.$$
\end{lemma}
\begin{proof}  Suppose $\cat M_\sharp =k$. Then $G^k_\sharp$  admits a section $s_\sharp$. We show that $G^k_1=G^k_{M_1}$ admits a section.

We will identify the $(n-1)$-skeleton of $M_1$ with a subspace of $M_\sharp$.  
The projection $q\co M_\sharp \to M_1$ gives rise to a map of Ganea fibrations $q_G\co G^k_\sharp\to G^k_1$. In particular, $q_G\circ s_\sharp$ defines a section of $G^k_1$ over the $(n-1)$-skeleton of $M_1$. Let $\kappa_1$ be the obstruction to extending this section over $M_1$. 
Then $q^*\kappa_1$ is the obstruction to a lift of $q$ with respect to the Ganea fibration $p:G^k_1\to M_1$. The section $s_\sharp$ defines such a lift.
Hence $q^*\kappa_1=0$. Since $M_2$ is orientable, $q$ is a map of degree one. By  Lemma~\ref{l:6}, $\kappa_1=0$.  
\end{proof}

Lemma~\ref{l:1.2} implies that if both $M_1$ and $M_2$ are orientable, then $\cat M_\sharp$ is bounded below by $\cat M_1$ and $\cat M_2$. Thus, Lemma~\ref{l:1.2} completes the proof of Theorem~\ref{th:1}. Here is  a refinement of Theorem~\ref{th:1}.
\begin{proposition}\label{ref}
Suppose that one of the closed connected manifolds $M_1$ and $M_2$ is orientable. Then $$\cat M_1\sharp M_2=\max\{\cat M_1,\cat M_2\}.$$
\end{proposition}
\begin{proof} Without loss of generality we may assume that $M_2$ is orientable.
We need to prove the inequality $\cat M_1\sharp M_2\ge\max\{\cat M_1,\cat M_2\}$ when $M_1$ is not orientable. In view of Lemma~\ref{l:1.2} it suffices to consider the case when $\cat
M_2>\cat M_1$. By the homotopy lifting property for any covering map $f:X\to Y$, $\cat X\le\cat Y$.
We consider the 2-fold covering $\mu:\tilde M_1\to M_1$ and the induced
covering $\nu: \tilde M_1\sharp(M_2\sharp M_2)\to M_1\sharp M_2$. Then $$\cat M_1\sharp M_2\ge\cat\tilde M_1\sharp(M_2\sharp M_2)=\max\{\cat\tilde M_1,\cat(M_2\sharp M_2)\}=\cat M_2$$ by the covering inequality, the connected sum equality (Theorem~\ref{th:1}) for orientable manifolds, and the assumption $\cat
M_2>\cat M_1$.
\end{proof}
\begin{problem}
Does the inequality $\cat M_1\sharp M_2\ge\max\{\cat M_1,\cat M_2\}$ hold true where both manifolds are  nonorientable ?
\end{problem}
Using the idea of the proof of Proposition~\ref{ref} one can reduce this problem to the case $M_1=M_2$.

\section{Proof of Theorem~\ref{th:0.2}}\label{s:3}

\begin{proposition}\label{ass}
Let $D$ be a closed ball in a closed connected orientable $n$-manifold $M$ and let $\overset\circ M$ denote $M\setminus Int D$.
Then the inclusion homomorphism $$j^*:H^k(M;G)\to H^k(\stackrel\circ{M};G)$$ is an isomorphism for $k\le n-1$
for any coefficient group $G$.
\end{proposition}
\begin{proof}
The exact sequence of  a pair 
$$
H^{k+1}(S^n;G) \stackrel\delta\leftarrow H^k(\stackrel\circ{M};G) \xleftarrow{j^*} H^k(M;G) \leftarrow H^k(S^n;G)
$$
immediately implies that $j^*$ is an isomorphism for $k<n$. For $k=n-1$ it suffices to show that $\delta$ is trivial. It follows from the fact
 $i^*=0$ in the following commutative diagram
$$
\begin{CD}
H^n(M,\stackrel\circ{M};G) @<\delta<< H^{n-1}(\stackrel\circ{M};G) \\
@V\cong VV @Vi^*VV\\
H^n(D^n,\partial D^n;G) @<\cong<< H^{n-1}(\partial D^n;G).\\
\end{CD}
$$
The triviality of $i^*$ follows from the exact sequence of a pair
$$
\to H^{n-1}(\stackrel\circ{M};G) \stackrel{i^*}\rightarrow H^{n-1}(\partial D^n;G)\stackrel{\delta'}\rightarrow H^n(M;G)\to .
$$
 We note that $\delta'$ as the connecting homomorphism in a Puppe exact sequence is a the composition 
$$
H^{n-1}(\partial D^n;G)\stackrel\Sigma\to H^{n}(S^n;G)\stackrel{q^*}\to H^n(M;G)
$$
of the suspension isomorphism and the homomorphism induced by the quotient map $q:M\to M/\stackrel\circ{M}$.
The homomorphism $q^*$ is an isomorphism due to orientability of $M$. Thus $\delta'$ is an isomorphism and hence $i^*$ is trivial.
\end{proof}
\begin{remark} The above proposition does not hold true for local coefficients. Let $\underline\Z$ be $\Z_2$-module obtained from the integers by the involution 
$1\to -1$. Let $M=\R P^3$, then $\stackrel\circ{M}=\R P^2$. The inclusion homomorphism
$H^2(\R P^3;\underline\Z)\to H^2(\R P^2;\underline\Z)$ is not an isomorphism since by the Poincare duality $H^2(\R P^3;\underline\Z)=H_1(\R P^3;\underline\Z)=H_1(\Z_2,\underline \Z)=0$ (see~\cite[Proposition 2.5.1]{Dr}) and $H^2(\R P^2;\underline\Z)=H_0(\R P^2;\underline\Z\otimes\underline \Z)
=H_0(\R P^2;\Z)=\Z$. 
\end{remark}

\begin{proposition}\label{product}
For any finite complex $X$ and a closed connected orientable manifold $M$ the inclusion homomorphism
$$
\nu:H^m(X\times M;G) \to H^m(X\times\stackrel\circ{M};G)
$$
is surjective for all $m$ and any coefficient group $G$
\end{proposition}
\begin{proof}
By the Kunneth formula we obtain the commutative diagram
\[
\begin{CD}
\underset{k+l=m}\bigoplus H^k(X)\otimes H^l(M;G) @>>> H^m(X\times M;G) @>>>\underset{k+l=m+1}\bigoplus Tor(H^k(X),H^l(M;G))\\
@V\oplus(1\otimes j^*)VV @V\nu VV @V\oplus(1\ast j^*)VV\\
\underset{k+l=m}\bigoplus H^k(X)\otimes H^l(\stackrel\circ M;G) @>>> H^m(X\times\stackrel\circ M;G) @>>> \underset{k+l=m+1}\bigoplus Tor(H^k(X),H^l(\stackrel\circ M;G))\\
\end{CD}
\]
By Proposition~\ref{ass}, $j^*$ is either an isomorphism or  a homomorphism  with zero range. Therefore each of the homomorphisms $1\otimes j^*$ and $1\ast j^*$ is either an isomorphism or has zero range. This implies that both $\oplus(1\otimes j^*)$ and $\oplus(1\ast j^*)$ are surjective.
The Five Lemma implies that $\nu$ is surjective.
\end{proof}

\begin{corollary}\label{cor}
The inclusion homomorphism $H^m(M_1\times M_2;G)\to H^m(\stackrel\circ{M}_1\times\stackrel\circ{M}_2;G)$ is surjective for all $m$
and any coefficient group $G$.
\end{corollary}
\begin{proof}
We apply Proposition~\ref{product} twice
$$
H^m(M_1\times M_2;G)\to H^m({M}_1\times\stackrel\circ{M}_2;G)\to H^m(\stackrel\circ{M}_1\times\stackrel\circ{M}_2;G) .
$$
\end{proof}
\begin{corollary}\label{forgetting}
The forgetting homomorphism $$j^*:H^m(M_1\times M_2,\stackrel{\circ}{M}_1\times\stackrel{\circ}{M}_2 ;G)\to H^m(M_1\times M_2;G)$$ is injective for all $m$, any closed connected orientable $n$-manifolds $M_1$ and $M_2$, and any coefficient group $G$.
\end{corollary}
\begin{proof}
The result follows from the exact sequence of the pair $(M_1\times M_2,\stackrel{\circ}{M}_1\times\stackrel{\circ}{M}_2)$ and Corollary~\ref{cor}. 
\end{proof}

\begin{proposition}\label{max}
For any simply connected closed orientable  $n$-manifold $M_1$ and any closed connected manifold $M_2$, $$ \TC(M_1\sharp M_2)\ge \TC(M_1).$$
\end{proposition}
\begin{proof}
Suppose that there is a section $s_\sharp:M_\sharp\to \Delta^k_\sharp$. We need to show that there is a section $s_1:M_1^2\to\Delta^k=\Delta^k_{M_1}$.
Let $D\subset M_1$ be the $n$-ball whose interior is cut off under the connected sum operation and let $\stackrel\circ{M}_1=M_1\setminus Int D$.
Let $f:M_1\sharp M_2\to M_1$ be a map which is the identity on $\stackrel\circ{M}_1$ and which sends the complement to $\overset\circ{M}_1$ to $D$.
The map $f$ induces a map of fibrations $\bar f:\Delta^k_\sharp\to \Delta^k$. Then $$s'=\bar f\circ s_\sharp\circ((f\times f)|_{\stackrel\circ{M}_1\times\stackrel\circ{M}_1})^{-1}:\stackrel\circ{M}_1\times\stackrel\circ{M}_1\to\Delta^k$$ is a section over $\stackrel\circ{M}_1\times\stackrel\circ{M}_1$.
We extend the section $s'$ to $M_1\times M_1$ by induction. 

Assume that the extension is already constructed on the $(m-1)$-skeleton of $M_1\times M_1$.
Let $\kappa\in H^m(M_1\times M_1,\stackrel\circ{M}_1\times\stackrel\circ{M}_1;G)$ be the obstruction for extension it to the $m$-skeleton.
Here we use the assumption that $M_1$ is simply connected to have the obstruction in cohomology group with constant coefficients.
Then the image $j^*(\kappa)\in H^m(M_1\times M_1;G)$ is the obstruction to the existence of a section over the $m$-skeleton of $M_1\times M_2$
(with a freedom to change it on the
$(m-1)$-skeleton). Since this  obstruction is natural, $(f\times f)^*(j^*(\kappa))$ is the obstruction to a lift of the map $f\times f$. Since the section $s_\sharp$ defines such a lift,
it follows that $(f\times f)^*(j^*(\kappa))=0$. Since $f\times f$ is a map of degree one, $(f\times f)^*$ is injective (see Lemma~\ref{l:6}). Hence $j^*(\kappa)=0$.
By Corollary~\ref{forgetting}, $\kappa=0$.
\end{proof}

\section{Proof of Theorem~\ref{th:0.3}}\label{s:4}

Recall that we denote the connected sum $M_1\sharp M_2$ by $M_\sharp$ and the pointed sum $M_1\vee M_2$ by $M_\vee$. There exists a fibration 
$p_\#:\Delta^k_\sharp\to M_\sharp^2$ such that $\TC(M_\sharp)\le k$ if and only if the fibration $\Delta^k_\sharp$ has a section. Similarly there is a fibration $p_\vee:\Delta^k_\vee\to M_\vee^2$ similarly related to $\TC(M_\vee)\le k$. The fibers of the fibrations $\Delta^k_\sharp$ and $\Delta^k_\vee$ are the joins $*^{k+1}\Omega(M_\sharp)$ and $*^{k+1}\Omega(M_\vee)$ of $(k+1)$ copies of loop spaces respectively. We will denote the inclusion $M_\sharp\to M_\vee$ by $j$. It induces a map $j_\Delta$ of the total spaces of fibrations $\Delta^k_\sharp\to \Delta^k_\vee$ and the map $j_\Delta'$ of the fibers of the fibrations. 

\begin{lemma}\label{l:12}  Under conditions of Theorem~\ref{th:0.3} the map $j_\Delta'$ is a $2n$-equivalence. 
\end{lemma}
\begin{proof}
By Proposition~\ref{p:1}, the map $\Omega(M_\sharp)\to \Omega(M_\vee)$ is an $(n-2)$-equivalence. 
Hence, by Proposition~\ref{p:2}, the map $j_\Delta'$ is $(kr+k+n-2)$-equivalence, which under hypotheses of Theorem~\ref{th:0.3} implies that $j_\Delta'$ is a $2n$-equivalence. 
\end{proof}


\begin{proposition}\label{l:2.1} Under the assumption of Theorem~\ref{th:0.3}, given a section $s_\vee$ of the fibration $\Delta^k_\vee$,  the fibration $\Delta^k_\sharp$ also admits a section. In particular, $\TC M_\sharp\le \TC M_\vee$.
\end{proposition}
\begin{proof}  Since manifolds with $\TC=1$ are spheres, the case $k=1$ is vacuous and we may assume in what follows that $k\ge 2$. 
We consider the pull-back diagram
\[
\begin{CD}
\Delta^k_\sharp @>h>> E @>j'>> G^k_\vee\\
@. @Vp'VV @Vp_\vee VV\\
@. M_\sharp^2 @>j>> M_\vee^2\\
\end{CD}
\]
where $p'\circ h=p_\sharp$ and $j'\circ h=j_\Delta$.  The section $s_\vee$ defines a section $s':M_\sharp^2\to E$. Let $F_\sharp$ and $F_\vee$ denote
the fibers of the fibrations $\Delta^k_\#$ and $\Delta_\vee$ respectively. Lemma~\ref{l:12} and the Five Lemma applied to the diagram
\[
\begin{CD}
\pi_{i+1}(M^2_\sharp) @>>>\pi_{i}(F_\sharp) @>>>\pi_i(\Delta^k_\#) @>>>\pi_i(M_\#^2) @>>>\pi_{i-1}(F_\sharp)\\
@V=VV @V(j'_\Delta)_*VV @Vh_*VV @V=VV @V(j'_\Delta)_*VV\\
\pi_{i+1}(M^2_\sharp) @>>>\pi_{i}(F_\vee) @>>>\pi_i(E) @>>>\pi_i(M_\sharp^2) @>>>\pi_{i-1}(F_\vee)\\
\end{CD}
\]
imply that $h$ is a $2n$-equivalence. Therefore the map $s'$ admits a lift $s''$ with respect to $h$. Then $s''$ is a section of $\Delta^k_\sharp$.
\end{proof}

\section{Examples}\label{s:5}

The idea of reducing the computation of $\cat$ and $\TC$ invariants of manifolds to the case of prime manifolds has been used by several authors in the past. In this section we list examples of computations in low dimensional cases.

\subsection*{The category of surfaces} Except for $S^2$, every closed surface is a connected sum of a finite number of tori  $T$ and a finite number of $\R P^2$. By the cup-length estimate $\cat(T)=\cat(\R P^2)=2$. By Theorem~\ref{th:1}, then, the Lusternik-Schnirelmann category of any closed surface except for $S^2$ is $2$, which is a well-known result.

\subsection*{The topological complexity of surfaces} Since the torus $T$ is a topological group, $\TC(T)=2$. 
Topological complexity of surfaces of genus $>1$ is $4$~\cite{Fa}. Since $\R P^2$ admits an immersion into $\R^3$, the topological complexity of $\R P^2$ is $3$, see \cite{FTY}.  The topological complexity of all other non-orientable surfaces is 4~\cite{CoV} (for surfaces of genus $> 3$ see~\cite{Dr}).

\subsection*{The category of $3$-manifolds}  The Lusternik-Schnirelmann category of $3$-manifolds is computed in \cite{GG}. The authors essentially used a decomposition of a $3$-manifold into a connected sum of prime manifolds.  

\subsection*{The category of simply connected $4$-manifolds} 
The category of simply connected 4-complexes does not exceed 4. Then the cup-length estimate yields $\cat(\C P^2)=2=\cat(-\C P^2)$ and $\cat(S^2\times S^2)=2$. Every smooth closed simply connected $4$-manifold is homotopy equivalent to the connected sum of a manifold with one of the three mentioned manifolds.  Consequently for every smooth closed simply connected $4$-manifold except for $S^4$, the category is $2$. 

\subsection*{The topological complexity of simply connected $4$-manifolds}  As in the case of the Lusternik-Schnirelmann category, the simple connectedness and the cup-length estimate implies that the topological complexity of $\C P^2, -\C P^2$ and $S^2\times S^2$ is $4$. Consequently for every smooth closed simply connected $4$-manifold except for $S^4$, the topological complexity is $4$.

\end{document}